\documentclass[11pt]{amsart}

\DeclareMathOperator{\Ind}{Ind}
\DeclareMathOperator{\Mor}{Mor}\DeclareMathOperator{\Mod}{Mod}
\DeclareMathOperator{\Res}{Res}

\DeclareMathOperator{\Spec}{Spec}

\usepackage[usenames,dvipsnames]{color}

\usepackage{amsmath,amssymb,amsthm,graphics,amscd,mathrsfs}
\usepackage{longtable}
\usepackage{cite}
\usepackage{color}
\usepackage{graphicx}
\usepackage{epsf}
\usepackage{fancyhdr,hyperref}
\usepackage{lscape}
\hypersetup{colorlinks=true,citecolor=Purple}

  \renewenvironment{thebibliography}[1]{
    \begin{oldthebibliography}{#1}
      \setlength{\parskip}{0ex}
      \setlength{\itemsep}{0ex}
  }
  {
    \end{oldthebibliography}
  }
  
\setlongtables
\begin{document}

\newcounter{rownum}
\setcounter{rownum}{0}
\newcommand{\ab}{\addtocounter{rownum}{1}\arabic{rownum}}

\newcommand{\x}{$\times$}
\newcommand{\bb}{\mathbf}

\newcommand{\R}{\mathrm{R}}
\newcommand{\A}{\mathbb{A}}
\newcommand{\RR}{\mathscr{R}}
\newcommand{\G}{\mathscr{G}}
\newcommand{\hra}{\hookrightarrow}
\newcommand{\sss}{\mathrm{ss}}
\newtheorem{lemma}{Lemma}[section]
\newtheorem{theorem}[lemma]{Theorem}
\newtheorem*{TA}{Theorem A}
\newtheorem*{TB}{Theorem B}
\newtheorem*{TC}{Theorem C}
\newtheorem*{CorC}{Corollary C}
\newtheorem*{TD}{Theorem D}
\newtheorem*{TE}{Theorem E}
\newtheorem*{PF}{Proposition E}
\newtheorem*{C3}{Corollary 3}
\newtheorem*{T4}{Theorem 4}
\newtheorem*{C5}{Corollary 5}
\newtheorem*{C6}{Corollary 6}
\newtheorem*{C7}{Corollary 7}
\newtheorem*{C8}{Corollary 8}
\newtheorem*{claim}{Claim}
\newtheorem{cor}[lemma]{Corollary}
\newtheorem{conjecture}[lemma]{Conjecture}
\newtheorem{prop}[lemma]{Proposition}
\newtheorem{question}[lemma]{Question}
\theoremstyle{definition}
\newtheorem{example}[lemma]{Example}
\newtheorem{examples}[lemma]{Examples}
\newtheorem{algorithm}[lemma]{Algorithm}
\newtheorem*{algorithm*}{Algorithm}
\theoremstyle{remark}
\newtheorem{remark}[lemma]{Remark}
\newtheorem{remarks}[lemma]{Remarks}
\newtheorem{obs}[lemma]{Observation}
\theoremstyle{definition}
\newtheorem{defn}[lemma]{Definition}

  \def\hal{\unskip\nobreak\hfil\penalty50\hskip10pt\hbox{}\nobreak
  \hfill\vrule height 5pt width 6pt depth 1pt\par\vskip 2mm}

\renewcommand{\labelenumi}{(\roman{enumi})}
\newcommand{\Hom}{\mathrm{Hom}}
\newcommand{\Int}{\mathrm{int}}
\newcommand{\Ext}{\mathrm{Ext}}
\newcommand{\opH}{\mathrm{H}}
\newcommand{\D}{\mathscr{D}}
\newcommand{\soc}{\mathrm{Soc}}
\newcommand{\SO}{\mathrm{SO}}
\newcommand{\Sp}{\mathrm{Sp}}
\newcommand{\SL}{\mathrm{SL}}
\newcommand{\GL}{\mathrm{GL}}
\newcommand{\PGL}{\mathrm{PGL}}
\newcommand{\OO}{\mathcal{O}}
\newcommand{\Y}{\mathbf{Y}}
\newcommand{\X}{\mathbf{X}}
\newcommand{\diag}{\mathrm{diag}}
\newcommand{\End}{\mathrm{End}}
\newcommand{\tr}{\mathrm{tr}}
\newcommand{\Stab}{\mathrm{Stab}}
\newcommand{\red}{\mathrm{red}}
\newcommand{\Aut}{\mathrm{Aut}}
\renewcommand{\H}{\mathcal{H}}
\renewcommand{\u}{\mathfrak{u}}
\newcommand{\Ad}{\mathrm{Ad}}
\newcommand{\N}{\mathcal{N}}
\newcommand{\id}{\mathrm{id}}
\newcommand{\Z}{\mathbb{Z}}
\newcommand{\la}{\langle}\newcommand{\ra}{\rangle}
\newcommand{\gl}{\mathfrak{gl}}
\newcommand{\g}{\mathfrak{g}}
\newcommand{\F}{\mathbb{F}}
\newcommand{\m}{\mathfrak{m}}
\renewcommand{\b}{\mathfrak{b}}
\newcommand{\p}{\mathfrak{p}}
\newcommand{\q}{\mathfrak{q}}
\renewcommand{\l}{\mathfrak{l}}
\newcommand{\del}{\partial}
\newcommand{\h}{\mathfrak{h}}
\renewcommand{\t}{\mathfrak{t}}
\renewcommand{\k}{\mathfrak{k}}
\newcommand{\Gm}{\mathbb{G}_m}
\renewcommand{\c}{\mathfrak{c}}
\renewcommand{\r}{\mathfrak{r}}
\newcommand{\n}{\mathfrak{n}}
\newcommand{\s}{\mathfrak{s}}
\newcommand{\Q}{\mathbb{Q}}
\newcommand{\z}{\mathfrak{z}}
\newcommand{\pso}{\mathfrak{pso}}
\newcommand{\so}{\mathfrak{so}}
\renewcommand{\sl}{\mathfrak{sl}}
\newcommand{\psl}{\mathfrak{psl}}
\renewcommand{\sp}{\mathfrak{sp}}
\newcommand{\Ga}{\mathbb{G}_a}

\newenvironment{changemargin}[1]{%
  \begin{list}{}{%
    \setlength{\topsep}{0pt}%
    \setlength{\topmargin}{#1}%
    \setlength{\listparindent}{\parindent}%
    \setlength{\itemindent}{\parindent}%
    \setlength{\parsep}{\parskip}%
  }%
  \item[]}{\end{list}}

\parindent=0pt
\addtolength{\parskip}{0.5\baselineskip}

\title{Irreducible modules for pseudo-reductive groups}

\author[M.\  Bate]{Michael Bate}
\address
{Department of Mathematics,
University of York,
York YO10 5DD,
United Kingdom}
\email{michael.bate@york.ac.uk}

%\author[B.\ Martin]{Benjamin Martin}
%\address%[B.\ Martin]
%{Department of Mathematics,
%University of Aberdeen,
%King's College,
%Fraser Noble Building,
%Aberdeen AB24 3UE,
%United Kingdom}
%\email{b.martin@abdn.ac.uk}
%\author[G. R\"ohrle]{Gerhard R\"ohrle}
%\address
%{Fakult\"at f\"ur Mathematik,
%Ruhr-Universit\"at Bochum,
%D-44780 Bochum, Germany}
%\email{gerhard.roehrle@rub.de}

\author{David I. Stewart}
\address{School of Mathematics and Statistics,
Herschel Building,
Newcastle,
NE1 7RU, UK}
\email{david.stewart@ncl.ac.uk}

\begin{abstract}
We classify the irreducible representations of smooth, connected affine algebraic groups over a field, by tackling the case of pseudo-reductive groups. We reduce the problem of calculating the dimension for pseudo-split pseudo-reductive groups to the split reductive case and the pseudo-split pseudo-reductive commutative case. Moreover, we give the first results on the latter, including a rather complete description of the rank one case.
\end{abstract}
\maketitle
%{\small \tableofcontents}
\section{Introduction}
\label{sec:intro}
Let $k$ be a field and let $G$ be a smooth connected affine algebraic $k$-group. We are interested in the irreducible $k$-representations of $G$. Since the only irreducible representation of a unipotent $k$-group is the trivial representation $k$ itself, any normal unipotent subgroup of $G$ must act trivially on any irreducible representation of $G$. In particular, the $k$-unipotent radical $\RR_{u,k}(G)$---that is to say the largest smooth connected normal unipotent $k$-subgroup of $G$---acts trivially and so we may as well assume $\RR_{u,k}(G)=1$; in other words $G$ is a pseudo-reductive group. The main result of this paper is to classify the irreducible representations of $G$ in terms of those of a maximal torus, effectively completing a programme started in the fifties by Chevalley.

Pseudo-reductive groups have been the focus of a high degree of interest in recent years, due for the most part to the monograph \cite{CGP15} which gives a remarkably transparent structure theory. It says that almost all the time, $G$ is standard: that is, isomorphic to a certain type of systematic modification of Weil restrictions of connected reductive groups; the modification process involves changing a Cartan subgroup---which is typically far from a torus. For simplicity of exposition all our reductive groups are henceforth assumed to be connected.

When $G$ is reductive and split, the representation  theory of $G$ over arbitrary fields is rather extensive; the reader is referred to \cite{Jan03} to see this in all its glory, but we mention some highlights. Firstly, there is, due to Chevalley, a parametrisation of the simple representations by dominant weights, with such representations arising as the socles of certain universal induced modules which are defined over $\Z$---the latter have an elegant formula for their characters and dimensions courtesy of Weyl. If the characteristic of $k$ is $0$, these induced modules \textit{are} irreducible, but even when they are not, there are effective methods of calculating the characters of their simple socles in many cases, using the Anderson--Jantzen sum formula, and informed by the so-called alcove geometry induced by the affine Weyl group (in particular the Linkage Principle). These methods have been implemented algorithmically by Frank L\"ubeck \cite{Lub01} and thousands of characters (in arbitrary characteristic) are now available. Furthermore, when the characteristic is huge relative to the root system, it is a result of a number of authors that Lusztig's character formula holds, relating the characters of simple modules to those of the induced modules via Kazhdan--Lusztig polynomials; technically, this gives information only about the principal block, but the remaining characters can be deduced by use of Jantzen's translation functors and Steinberg's tensor product theorem. It should be mentioned that work of G.~Williamson \cite{Wil17} tells us that the characteristic must be at least exponential in the rank of simple factors for Lusztig's formula to hold, so that there remains a conceptual hole in the theory, but one which continues to be closed as time goes on; see \cite{RW18} for the latest developments, including a replacement conjecture.

%Additionally, let us mention that Tits has explained in \cite{Tit71} how one relates the irreducible representations of an arbitrary reductive group $G$ to a split reductive group and a certain division algebra .

The representation theory of split reductive groups is all predicated on the commutative case: a split reductive commutative group is simply a product of copies of the multiplicative group of the field and its representations are all well-known to be semisimple, being just sums of one-dimensional weight spaces. Since this is completely false in the context of commutative pseudo-reductive groups---and their classification is thought to be out of reach---it has been expected that their representation theory should be intractable. However, in the case where $G$ is pseudo-split---that is, it contains a split maximal torus---we are able to classify the simple representations by dominant weights and reduce the case of giving a dimension formula to understanding the commutative pseudo-reductive case together with the reductive case. 

The possibility for a breakthrough owes itself to the following crucial theorem, \cite[Thm.~3.4.6]{CGP15} (or the simpler proof of \cite[Thm.~5.4.4]{CPSurv}):
\begin{theorem}[Conrad--Gabber--Prasad] Let $G$ be a pseudo-split pseudo-reductive $k$-group with split maximal torus $T$. Then $G$ has a Levi $k$-subgroup $M$ containing $T$.
%(and $M$ is the unique Levi subgroup containing $T$).
\end{theorem}

Recall that $M$ is a \emph{Levi subgroup} of $G$ if $M$ is reductive and $G_{\bar k}=M_{\bar k}\ltimes \RR_u(G_{\bar k})$, where $\RR_{u}(G_{\bar k})$ is the unipotent radical of $G_{\bar k}$. Our main theorem constructs a correspondence between the irreducible $G$-modules and the irreducible $M$-modules and reduces a description of the dimension of irreducible $G$-modules to that of $M$-modules and $C$-modules where $C$ is a Cartan subgroup of $G$. 

\begin{theorem}\label{maintheorem}
Let $G$ be a pseudo-split pseudo-reductive group with Cartan subgroup $C$ containing a split maximal torus $T$. Let $M$ be a Levi subgroup of $G$ containing $T$. Then the isomorphism classes of irreducible representations of $G$ are in $1$-$1$ correspondence with the dominant weights of $M$. If $X(T)_+$ denotes the set of dominant weights for $T\subseteq M$, then for $\lambda\in X(T)_+$ we denote by $L_G(\lambda)$ the corresponding irreducible representation. On restriction, $L_G(\lambda)$ is $M$-isotypic and semisimple. Furthermore,
\[\dim L_G(\lambda)=\dim L_M(\lambda)\cdot\dim L_C(\lambda).\]\end{theorem}
For the dimension formula, note that since a Cartan subgroup $C$ contains a split maximal torus $T$ which must be its Levi subgroup, the first part of the theorem guarantees a representation $L_C(\lambda)$ unique up to isomorphism for any weight $\lambda\in X(T)$.

As mentioned above, a complete description of $\dim L_M(\lambda)$ is  thought to be out of reach for $p$ small compared to the root system of $M$, though at least there are algorithms that in principle compute any given example. By contrast, there are no results at all on $\dim L_C(\lambda)$. We can give a formula for the latter in the case $C=\R_{k'/k}(\Gm)$ for $k'$ a finite non-zero reduced purely inseparable $k$-algebra (Theorem \ref{thm:anykalgebra}). Here is the simpler version of this theorem when $k'$ is a purely inseparable field extension. In order to state the result, we need some notation: 
let $k'/k$ be a purely inseparable extension of fields of degree $q=p^r$ and let $\lambda \in\Z$.
Then we let $k'(\lambda)$ denote the subfield of $k'$ generated by $k$ and 
$(k')^\lambda$.
Note that if $\lambda$ is coprime to $p$ and $k'/k$ is purely inseparable, then the kernel of the 
group homomorphism 
$x\mapsto x^\lambda$ on $(k')^*$ is contained in $k$.
The fundamental theorem of homomorphisms now implies that any element of $(k')^*$ lies in the product of $k$ and the image of this map, so in this case $k'(\lambda) = k'$.
More generally, one can see that if we write $\lambda = p^{\nu_p(\lambda)}\mu$ with $\mu$ coprime to $p$ then $k'(\lambda) = k'(p^{\nu_p(\lambda)})$. 

\begin{theorem}\label{gmfield}
Let $k'/k$ be a purely inseparable extension of fields of degree $q=p^r$ and let $C=\R_{k'/k}(\Gm)$. 
Then for any $\lambda \in \Z$ we have \[\dim(L_C(\lambda))=[k'(\lambda):k].\]
\end{theorem}

This is a good moment to point out that if $G$ is a general pseudo-split pseudo-reductive group, most simple $G$-modules are not absolutely irreducible, in contrast with the split reductive case. If $V$ is an absolutely irreducible $G$-module then $V_{\bar k}$ is an irreducible $G_{\bar k}$-module and the Lie--Kolchin theorem implies that $\RR_u(G_{\bar k})$ must act trivially. For example, if $G=C$ is as in the theorem above, we have $G_{\bar k}\cong \Gm\times \RR_u(G_{\bar k})$, so that the simple $G_{\bar k}$-modules are $1$-dimensional, whereas this is true of simple $G$-modules if and only if $k'(\lambda)=k$. Indeed, for any pseudo-split, pseudo-reductive $G$, unless $\dim(L_C(\lambda))=1$, we have $L_G(\lambda)$ is not even absolutely semisimple, though it is absolutely indecomposable, since the socle of $L_G(\lambda)_{k'}$ is still irreducible for any extension of fields $k'/k$; see Remark \ref{abssemi} below.

For split reductive $G'$, as the action of $G'$ on $L_{G'}(\lambda)$ factors through the Frobenius map $F^{\nu_p(\lambda)}$, we have that $\R_{k'(\lambda)/k}(L_{G'}(\lambda))$ furnishes us with a $G$-module with the correct highest weight and the right dimension. Furthermore, the Weil restriction of a reductive $k'$-group $G'$ across a field extension $k'/k$ is pseudo-split if and only if $G'$ is split and $k'/k$ is purely inseparable (see \cite[A.5.15]{CGP15}). Hence, the following consequence of Theorem \ref{gmfield} is immediate.

\begin{cor} Let $k'/k$ be a purely inseparable extension of fields. If $G=\R_{k'/k}(G')$ for $G'$ split reductive, then $L_G(
\lambda)=\R_{k'(\lambda)/k}(L_{G'}(\lambda))$.\end{cor}

We finish this introduction by returning to the general question of classifiying the irreducible representations for 
an arbitrary smooth connected affine algebraic $k$-group.
First recall that the paper \cite{Tit71} describes how to relate the irreducible representations for a non-split reductive $k$-group $G$ to a split reductive group $G_{K}$ via Galois cohomology, where $K/k$ is an appropriate Galois extension. One associates isomorphism classes of irreducible representations to orbits of the Galois group $\Gamma=\mathrm{Gal}(K/k)$ on the dominant weights of a maximal torus $T$ of $G_{K}$.
The same programme is straightforward to apply in our situation, reducing the classification of irreducible modules for general pseudo-reductive groups to the pseudo-split case, and giving rise to the following theorem.

\begin{theorem}
Let $G$ be a smooth connected affine algebraic $k$ group, and let $G' = G/\RR_{u,k}(G)$ be its maximal pseudo-reductive quotient. 
Given a maximal torus $T'$ of $G'$, there is a finite Galois extension $K/k$ so that $G'_K$ is pseudo-split with split maximal torus $T'_K$.
The Galois group $\Gamma = \mathrm{Gal}(K/k)$ acts on the dominant weights $X(T'_K)_+$, and 
there is a one-one correspondence between $\Gamma$-orbits in $X(T'_K)_+$ and isomorphism classes of irreducible representations of $G$.

Moreover, if $V$ is an irreducible representation of $G$ corresponding to the $\Gamma$-orbit $\{\lambda_1,\ldots,\lambda_r\}$, then $V_K$
decomposes as a direct sum of the irreducible $G'_K$-modules $L_{G'_K}(\lambda_i)$,
each of which appears with the same multiplicity.
\end{theorem}

\section{Preliminaries}

We collect some basic material in this section. Our main references for the theory of algebraic groups are \cite{CGP15} and \cite{Jan03} and our notation will be kept consistent with those monographs. In particular all rings are commutative and unital.

For $k$ a ring, let $G$ be an affine algebraic $k$-group, that is a functor $\underline{k\mathrm{-Alg}}\to \underline{\mathrm{Grp}}$ which is represented by a finitely presented $k$-algebra $k[G]$, in other words $G(?)\cong \Hom_{k\mathrm{-Alg}}(k[G],?)$. Note that we do not insist that algebraic groups be smooth.

\subsection{$G$-modules}\label{sec:Gmods} Let $M$ be a $k$-module (possibly infinite-dimensional). Then we may define a group functor $M_a:\underline{k\mathrm{-Alg}}\to \underline{\mathrm{Grp}}$ so that $M_a(A)=M\otimes_k A$ inherits a group structure from the additive group on $A$. Note that, even when $k$ is a field, $M_a$ is only an algebraic group when $M$ is finite-dimensional. Recall that an action of $G$ on a $k$-functor $X$ is a morphism (i.e. a natural transformation) $\phi:G\times X\to X$ such that $\phi(A):G(A)\times X(A)\to X(A)$ is an action of the group $G(A)$ on $X(A)$ for each $k$-algebra $A$. In case $G$ acts on $M_a$ such that the action of $G(A)$ on $M(A)$ is $A$-linear for each $k$-algebra $A$, we say $M$ is a \emph{representation for $G$}, or more frequently in this paper, a \emph{$G$-module}. Equivalently, one may use the Hopf algebra structure on $k[G]$ to define a $G$-module $M$ to be a comodule for $k[G]$. These  definitions permit the possibility of working with infinite-dimensional modules, though if $V$ is a finite-dimensional $G$-module then it corresponds to a homomorphism $G\to \GL(V)$ of algebraic groups. Of course, if $k\to k'$ is a homomorphism of rings and $M$ is a $G$-module then $M_{k'}:=M\otimes k'$ acquires an action of the base change $G_{k'}$ of $G$ making it into a $G_{k'}$-module.

\begin{remark}If $G$ is smooth and $k$ is an algebraically closed field then one may more straightforwardly define a $G$-module to be a vector space $M$ over $k$ on which $G(k)$ acts rationally through $k$-linear maps. Here to act rationally means that if $g\in G(k)$ and $(v_i)_{i\in I}$ is a basis for $M$ then $g.v_i=\sum_{j\in J}f_{ji}(g)v_i$ for $f_{ji}\in k[G]$ with cofinitely many of the $f_{ij}$ being zero.\end{remark}

The collection of $G$-modules forms a category $G$-$\Mod$, with morphisms being $G$-equivariant $k$-linear maps. If $M$ and $M'$ are $G$-modules, the full collection of such morphisms is written $\Hom_{G}(M,M')$. An important fact \cite[I.2.10(7)]{Jan03} is that the $\Hom_G$ bifunctor commutes with base change across flat extensions, i.e. \begin{equation}\label{hombase}\Hom_G(V,W)\otimes k'\cong \Hom_{G_{k'}}(V\otimes k',W\otimes k').\end{equation}

If $G$ is flat, then it is an immediate consequence of the definitions that all $G$-modules are \emph{locally finite}, that is to say that for any $m\in M(k)$ there is a unique minimal finitely generated submodule $G$-submodule $kGm$ of $M$ containing $m$. It follows that all simple $G$-modules over a field are finite-dimensional. Furthermore, the category of $G$-modules is abelian.

One may consult \cite[\S I.2]{Jan03} for more details.

\subsection{Representations of $\Gm$} The $\Z$-defined group scheme $\Gm$ is the functor $\underline{\mathrm{Rng}}\to\underline{\mathrm{Grp}}$ which returns the group of units $R^\times$ of any ring $R$. (It is represented by the algebra $\Z[t,t^{-1}]$.) Let $k$ be a ring and let $W$ be a non-zero $(\Gm)_k$-module. If there is $\lambda\in\Z$ such that $a\cdot w=a^\lambda w$ for any $k$-algebra $A$, $w\in W(A)$ and $a\in(\Gm)_k(A)=A^\times$, then we say $W$ is a \emph{weight module} (of weight $\lambda$). More typically $k$ will be a field and so $W(k)$ will be a vector space over $k$, in which case we refer to it as a \emph{weight space} (of weight $\lambda$). By \cite[I.2.14(4)]{Jan03}, any $(\Gm)_k$-module $V$ is semisimple, breaking into a sum of $1$-dimensional irreducible weight spaces; the resulting weights are referred to as the weights of $V$. If $\lambda$ is a weight of $V$ then $V_\lambda$ is the sum of all submodules of $V$ which are weight modules of weight $\lambda$. When $k$ is a field, then an irreducible representation is a $1$-dimensional weight space. We denote by $k_\lambda$ a $1$-dimensional weight space of weight $\lambda$. We will usually abuse notation by identifying the character group $X(\Gm)$ with $\Z$.

In many cases it will be simpler to consider $\Gm$ as a $k$-group over some ring $k$ (which will usually be a field), in which case we will just write $\Gm$ in place of $(\Gm)_k$.

\subsection{Representations of unipotent groups} In this section let $k$ be a field. Recall that a $k$-group $U$ is \emph{unipotent} if it is isomorphic to a closed subgroup of the group $U_n$ of strictly upper triangular $n\times n$ matrices over $k$ for some $n$. In the case that $k$ is a field of characteristic $p$ and $U$ is a smooth $k$-group, for $U$ to be unipotent, it suffices for there to be some $e$ such that the $p^e$-map on $U$ factors through the identity. In order to apply the Lie--Kolchin theorem to a solvable group $G$, one needs $G$ to be split (which happens in the case $k$ is algebraically closed), but when $G$ is a unipotent $k$-group one can show, \cite[Exp.XVII, Prop.~3.2]{SGA3}:
\begin{prop}\label{unipsimp}Let $U$ be any unipotent $k$-group. Then the only simple $U$-module is the $1$-dimensional trivial module, $k$.\end{prop}

\subsection{Induction} 
We construct simple modules by induction. 
The archetypal use of induction for reductive algebraic groups is of simple modules for a maximal torus, lifted to a Borel subgroup $B$ and induced to $G$. Such modules are then finite-dimensional since $G/B$ is a projective variety. The reader is warned that the induced modules we consider are generally infinite-dimensional.

The essential definition is this: Let $k$ be a unital ring and $M$ be an $H$-module for $H$ a closed flat subgroup scheme of the flat $k$-group scheme $G$ and let $M_a$ be the underlying $k$-group functor of $M$. 
Then from \cite[I.3.3]{Jan03} we have
\[\Ind_H^G(M)=\{f\in\Mor(G,M_a)\mid f(gh)=h^{-1}f(g)\text{ for all }g\in G(A),h\in H(A)\text{ and all }k\text{-algebras }A\},\]
is a $G$-module via $(g_1\cdot f)(g)=f(g_1^{-1}g)$. Of course, if $H=1$ is the trivial subgroup of $G$ we have $\Ind_H^G(k)=k[G]$ is the co-ordinate algebra of $G$, considered as a left $G$-module.

A key feature of induction is Frobenius reciprocity. For a $G$-module $N$ and $H$-module $M$, we have
\begin{equation}\Hom_G(N,\Ind_H^G(M))\cong \Hom_H(\Res^{G}_H(N),M),\label{frobreq}\end{equation}
where $\Res^G_H(N)=N|_H$ is the obvious $H$-module obtained by restriction.

If $G$ is unipotent and $k$ is a field then, then Prop.~\ref{unipsimp} implies that $k$ is the only simple module, and taking $H=1$, the above equation gives
\begin{equation}\label{unipindecomp}k\cong \Hom_H(k,k)\cong \Hom_G(k,\Ind_1^G(k))=\Hom_G(k,k[G]).\end{equation}
Thus $k[G]$ has a unique simple module in its socle (and is therefore indecomposable). 

This argument can be run in reverse, so that if one shows an induced module has a simple socle then it will follow that there is exactly one simple $G$-module up to isomorphism which has an $H$-homomorphism to the $H$-module being induced.

Another fact we need is that induction commutes with base change, \cite[I.3.5(3)]{Jan03}. Let $k'$ be a flat $k$-algebra. Then we have for each $H$-module $M$ a canonical isomorphism
\begin{equation}\label{indextension}\Ind_H^G(M)\otimes k'\cong \Ind_{H_{k'}}^{G_{k'}}(M\otimes k').\end{equation}

Lastly, we recall the tensor identity, \cite[I.3.6]{Jan03}. Let $N$ be a $G$-module that is flat over $k$. For any closed flat subgroup scheme $H$ of $G$ and any $H$-module $M$ there is a canonical isomorphism of $G$-modules
\begin{equation}\label{tensorid}\Ind_H^G(M\otimes \Res_H^G(N))\cong \Ind_H^G(M)\otimes N.\end{equation}

%\subsection{Clifford's theorem}In this section $k$ will be a field.
%
%The following, which we use often, is \cite[I.6.16]{Jan03}. We will apply it only in the case that $G=\R_{k'/k}(\G_m)$ for $k'$ a non-zero finite reduced $k$-algebra, ensuring that the hypothesis of density holds. Note that pseudo-reductive groups can fail to have Zariski dense set of $k$-points even when they are pseudo-split and there are examples in \cite[\S11.3]{CGP15}. 
%
%\begin{prop}\label{clifford}Let $G$ be a $k$-group with normal subgroup $N$ and assume the set of $k$-points $G(k)$ is Zariski dense in $G$. Then for any $G$-module $M$, its $N$-socle $\soc_N(M)$ is a $G$-submodule of $M$ with $\soc_G(M)\subseteq \soc_N(M)$.\end{prop}
%
%If the $G$-module $M$ in the theorem is semisimple, then $M=\soc_G(M)=\soc_N(M)$ and so we recover the usual statement of Clifford's theorem that restriction of a semisimple module to a normal subgroup is semisimple again. 

\subsection{Weil restriction}\label{sec:Weil} Since the notion of Weil restriction is at the heart of the structure theory of pseudo-reductive groups, we recall some of the important features from \cite[\S A.5]{CGP15}.	 If $B\to B'$ is a finite flat map of Noetherian rings, and $X'$ a quasi-projective $B'$-scheme, one may define the Weil restriction $X:=\R_{B'/B}(X')$. Then $X$ is a $B$-scheme of finite type satisfying the universal property
\[X(A)=X'(B'\otimes_B A),\]
for $A$ any $B$-algebra.

A key fact is that Weil restriction is right adjoint to base change along $\Spec(B)\to \Spec(B')$. That is to say that there is a bijection
\begin{equation}\label{weiladj}\Hom_B(Y,\R_{B'/B}(X'))\cong \Hom_{B'}(Y_{B'},X'),\end{equation}
which is natural in $X'$ and the $B$-scheme $Y$. Two situations are particularly important. If $X'=Z_{B'}$ for a $B$-scheme $B'$ then taking $Y=Z$ in (\ref{weiladj}), one has the identity map on the right-hand side, giving a canonical map $Z\to \R_{B'/B}(X')$; \cite[A.5.7]{CGP15} implies that this map is a closed immersion provided $\Spec(B')\to \Spec(B)$ is surjective (which will be true if $B$ is a field and $B'$ is non-zero, since then $\Spec(B)$ is a single point). Conversely, if we take $Y=\R_{B'/B}(X')$ the identity map on the left-hand side corresponds to a canonical map $\mathsf q:\R_{B'/B}(X')_{B'}\to X'$; \cite[A.5.10]{CGP15} implies this map is is surjective on all $A$-points for $A$ a $B$-algebra provided $B$ is a field and $B'$ is a finite local $B$-algebra with a purely inseparable residue field over $B$.

In case $X'=G'$ is a $B'$-group, we find $G:=X$ is a $B$-group. When $B=k$ is a field, and $B'=k'$ is a nonzero finite reduced $k$-algebra, then $G'$ is pseudo-reductive whenever $G$ is reductive. If $G'$ is defined over $k$ and we choose a $k$-descent $H$ of $G'$, then the remarks above show that $H$ embeds as a canonical subgroup in $G$; in particular this holds in the case that $G'$ is split reductive, hence:
\begin{equation}\label{canonicalsub}\text{If $G$ is a split reductive $k$-group, then $G$ embeds as a canonical subgroup of $\R_{k'/k}(G_{k'})$.}\end{equation}

\section{Existence and uniqueness}\label{sec:exist}

Let $k$ be an imperfect field of characteristic $p$ and let $G$ be a pseudo-split, pseudo-reductive $k$-group. Then by \cite[Thm.~3.4.6]{CGP15}, $G$ has a Levi subgroup $M$, containing a split torus $T$. A choice of Borel subgroup $B_M\supseteq T$ in $M$ containing negative root groups defines a partial ordering on weights together with a set of dominant weights $X(T)_+$. Since everything is flat over $k$, we may apply all the results of the previous section. Moreover, by \cite[II.2.4]{Jan03}, each simple module for $M$ has a unique highest weight $\lambda\in X(T)_+$, and any such is isomorphic to $L_M(\lambda):=\soc_M(\Ind_{B_M}^M(\lambda))$.

Having fixed this notation, we prove essentially the same is true for $G$. Define $Q_G(\lambda):=\Ind_M^G(L_M(\lambda))$ and let $k'$ be the minimal field of definition of the unipotent radical of $G$, so that $G_{k'}\cong M_{k'}\ltimes \RR_{u,k'}(G_{k'})$. Therefore, the $M_{k'}$-module $L_{M_{k'}}(\lambda)$ inherits a $G_{k'}$-structure by allowing $U:=\RR_{u,k'}(G_{k'})$ to act trivially on $L_{M_{k'}}(\lambda)$; conversely $U$ acts trivially on any simple $G_{k'}$-module, so this structure is unique. Write $L_{G_{k'}}(\lambda)$ for this module. 

%\begin{lemma}\label{qglambda}The module $Q_G(\lambda)$ has a simple socle.\end{lemma}
%\begin{proof}We have $Q_G(\lambda)_{k'}\cong \Ind_{M_{k'}}^{G_{k'}}(L_{M_{k'}}(\lambda))$ by (\ref{indextension}). Now \cite[I.3.8(2)]{Jan03}, gives us that $Q_G(\lambda)_{k'}\cong k[U]\otimes L_{G_{k'}}(\lambda)$, where $U$ acts by left translation on $U$, and $M_{k'}$ acts through the conjugation action on $k[U]$ and as usual on $L_{G_{k'}}(\lambda)$. Now, by (\ref{{unipindecomp}}), $k[U]$ is indecomposable with simple socle. Suppose \begin{align*}0\neq\Hom&_{G_{k'}}(L_{G_{k'}}(\mu),Q_{G_{k'}}(\lambda))\\
%&\cong \Hom_{G_{k'}}(L_{G_{k'}}(\mu),k[U]\otimes L_{M_{k'}}(\lambda))\\
%&\cong \Hom_{M_{k'}}(L_{M_{k'}}(\mu),k[U]\otimes L_{M_{k'}}(\lambda))
%\end{align*} Then 
%
%we have trivially on $\Ind_{M_{k'}}^{G_{k'}}(k)$. Hence $Q_G(\lambda)_{k'}$ is a semisimple $M_{k'}$-module whose composition factors are all isomorphic to $L_{M_{k'}}(\lambda)$. However if $L_M(\mu)$ is a composition factor of $Q_G(\lambda)$ then as $L_M(\mu)_{k'}\cong L_{M_{k'}}(\mu)$ by \cite[II.2.9]{Jan03}, we see $\mu=\lambda$, showing that $Q_G(\lambda)$ is $M$-isotypic. Finally, there are no self-extensions of simple $M$-modules---\!\!\cite[II.2.12(1)]{Jan03}---so $Q_G(\lambda)$ must be $M$-semisimple also.\end{proof}

\begin{theorem}\label{existanduniq}The socle of $Q_G(\lambda)$ is a simple $G$-module $L_G(\lambda)$. Any simple module for $G$ is isomorphic to $L_G(\lambda)$ for precisely one $\lambda\in X(T)_+$.\end{theorem}
\begin{proof}Let $\lambda\in X(T)_+$. Define $L_G(\lambda):=\soc_G(Q_G(\lambda))=\soc_G(\Ind_M^G(L_M(\lambda)))$. Then $L_G(\lambda)$ is non-zero if and only if there is a module $V$ such that $\Hom_G(V,Q_G(\lambda))\neq 0$, which is true if and only if $0\neq\Hom_G(V,Q_G(\lambda))_{k'}\cong\Hom_{G_{k'}}(V_{k'},Q_G(\lambda)_{k'})$, using (\ref{hombase}).

By Frobenius reciprocity and (\ref{indextension}) we have \[\Hom_{G_{k'}}( L_{G_{k'}}(\lambda),Q_G(\lambda)_{k'})\cong \Hom_{M_{k'}}(L_{M_{k'}}(\lambda),L_{M_{k'}}(\lambda))\cong k'\] and in particular $Q_G(\lambda)\neq 0$. Furthermore, since $\RR_{u,k'}(G_{k'})$ acts trivially on any simple $G_{k'}$-module, any simple $G_{k'}$-module is isomorphic to $L_{G_{k'}}(\mu)$ for some $\mu$. In particular, we have $\dim_{k'}\Hom_{G_{k'}}(L_{G_{k'}}(\mu),Q_G(\lambda)_{k'})=\delta_{\lambda\mu}$, so $Q_G(\lambda)_{k'}$---thus also $Q_G(\lambda)$---is indecomposable, with simple socle. Therefore $L_G(\lambda)=\soc_G(Q_G(\lambda))$ is a simple $G$-module as required. Furthermore, since $Q_G(\lambda)_{k'}$ has simple socle $L_{G_{k'}}(\lambda)$, an isomorphism $L_G(\lambda)\cong L_G(\mu)$ implies $L_{M_{k'}}(\lambda)\cong L_{M_{k'}}(\mu)$, so that $\lambda=\mu$.

Finally take any simple $G$-module $V$. This is finite-dimensional by local finiteness and so $\Res^G_M(V)$ has a simple $M$-quotient isomorphic to $L_M(\lambda)$, say, with $\lambda\in X(T)_+$. By Frobenius reciprocity, we get a homomorphism $V\to Q_G(\lambda)$, giving an isomorphism $V\cong L_G(\lambda)$. 
\end{proof}

\begin{remarks} (i). The proof of the theorem actually shows that for any affine algebraic group $G$ over a field $k$ (not necessarily connected or smooth), if $G$ admits a subgroup $M$ such that for some field extension $k'/k$, we have $G_{k'}\cong M_{k'}\ltimes U$ for some unipotent $k'$-group $U$, then its isomorphism classes of irreducible representations are in one-to-one correspondence with those of $M$.

(ii). The relationship between $L_G(\lambda)_{k'}$ and $L_{G_{k'}}(\lambda)$ is not as straightforward as the notation might suggest; in particular, 
$L_G(\lambda)_{k'}$ is not even semisimple in general. See Remark \ref{rems:notsemisimple}(i) below.
\end{remarks}

\section{Dimension formula and restriction to $M$}

Keep the notation of the previous section. Let $C=Z_G(T)$ be a Cartan $k$-subgroup of $G$. Then by Theorem \ref{existanduniq}, there is a unique simple $C$-module for any $\lambda\in X(T)$. (Note that $C$ is commutative so all weights are dominant for $C$.) Here we prove the following.

\begin{theorem}\label{dimformula}Let $\lambda\in X(T)_+$. Then \[\dim L_G(\lambda)=\dim L_M(\lambda)\cdot \dim L_C(\lambda).\]\end{theorem}

Following \cite[\S2.1]{CGP15}, as $G$ is pseudo-split, we may take $\lambda$ a regular cocharacter with $C=Z_G(\lambda)$. We may define $B:=P_G(\lambda)$   as the subgroup whose $A$-points for any $k$-algebra $A$ is the collection \[P_G(\lambda)(A):=\{g\in G(A)\mid \lim_{t\to 0}\lambda(t)g\lambda(t)^{-1}\text{ exists}\}.\] Then $B$ is a minimal pseudo-parabolic subgroup or \emph{pseudo-Borel subgroup}, and we may assume that it corresponds  to the negative roots. 
Define also $B^+:=P_G(\lambda^{-1})$, the corresponding opposite pseudo-Borel. We have the decompositions $B:=C\ltimes U$ and $B^+:=C\ltimes U^+$ where $U:=U_G(\lambda)$ with
\[U_G(\lambda)(A):=\{g\in G(A)\mid \lim_{t\to 0}\lambda(t)g\lambda(t)^{-1}=1\},\]
and $U^+:=U_G(\lambda^{-1})$; indeed $U=\RR_{u,k}(B)$ and $U^+=\RR_{u,k}(B^+)$. Furthermore, $B\cap B^+=C$, by inspection.

The commutativity of $C$ implies that any weight space for $T\subseteq C$ is stable under $C$. From Theorem \ref{existanduniq} we therefore must have a submodule isomorphic to $L_C(\mu)$ in the $C$-module $L_G(\lambda)_\lambda$ for some $\mu\in X(T)$. It is straightforward to see that we must have $\lambda=\mu$. Thus the $\lambda$-weight space of $L_G(\lambda)$ is dimension at least $L_C(\lambda)$. We aim to exhibit a $G$-module whose highest weight is $\lambda$ and whose $\lambda$-weight space is dimension $L_C(\lambda)$. Thence we will observe that $\Res^G_M(L_G(\lambda))$ is $M$-semisimple and isotypic and the dimension formula will follow. For this we follow the  programme of \cite[\S II.2]{Jan03}.

For a $G$-module $V$, the fact that $U$ and $U^+$ are unipotent implies that $\soc_U(V)$ and $\soc_{U^+}(V)$ are non-zero modules for $U$ with trivial composition factors. Thus $V^U$ and $V^{U^+}$ 
(the subspaces of $U$- and $U^+$-fixed vectors) are both non-zero.

As $C$ normalises $U$ (resp.~$U^+$), $V^U$ (resp.~$V^{U^+}$) is a $B$-submodule (resp.~$B^+$-submodule) of $V$ on which $U$ (resp.~$U^+$) acts trivially. A simple $B$-submodule $W$ of $V^U$ (resp. $B^+$-submodule of $V^{U^+}$), restricts to a simple $C$-module $W|_C\cong L_C(\lambda)$ (resp.~$W|_C\cong L_C(\lambda'))$ guaranteed by Theorem \ref{existanduniq}; we denote the isomorphism class of $W$ by $L_B(\lambda)$ (resp.~$L_{B^+}(\lambda')$). We have thus: 

\begin{center}
there are $\lambda,\lambda'\in X(T)$ with $\Hom_B(L_B(\lambda),V)\neq 0\neq \Hom_{B^+}(L_B(\lambda'),V)$.
\end{center}

If $V$ is finite-dimensional (for example if $V$ is simple), then we may apply the above to $V^*$ and dualise to get 

\begin{center}
there are $\lambda,\lambda'\in X(T)$ with $\Hom_B(V,L_B(\lambda))\neq 0\neq \Hom_{B^+}(V,L_{B^+}(\lambda'))$.
\end{center}

Now, using Frobenius reciprocity \eqref{frobreq}, we get

\begin{lemma}\label{thereisalambda}
If $\dim V<\infty$, then there are $\lambda,\lambda'\in X(T)$ with 
\[\Hom_G(V,\Ind_B^G(L_B(\lambda)))\neq 0\neq \Hom_G(V,\Ind_{B^+}^G(L_{B^+}(\lambda')))\]
\end{lemma}

Denote the module $\Ind_B^G(L_B(\lambda))$ by $H^0(\lambda)$.

From \cite[Prop. 2.1.8(3)]{CGP15} we have:
\begin{equation}\label{bigcell}
U^+B\text{ and }UB^+\text{ are dense in }G.
\end{equation}

With this in hand, we can prove:

\begin{prop}\label{highweightprop} 
Let $\lambda\in X(T)$ with $H^0(\lambda)\neq 0$.
\begin{itemize}
\item[(a)] We have $\dim H^0(\lambda)^{U^+}=\dim L_B(\lambda)$ and $H^0(\lambda)^{U^+}=H^0(\lambda)_\lambda$.

\item[(b)] Each weight $\mu$ of $H^0(\lambda)$ satisfies $w_0(\lambda)\leq \mu\leq \lambda$,
where $w_0$ denotes the longest element in the Weyl group $W$.
\end{itemize}
\end{prop}

\begin{proof} Recall that \[H^0(\lambda)=\{f\in \Mor(G,L_C(\lambda))|f(gb)=b^{-1}f(g)\text{ for all }g\in G(A), b\in B(A)\text{ and all }A\}.\]

The action of $G$ is given by left translation. Since $U^+$ acts trivially on $H^0(\lambda)^{U^+}$ and $U \subseteq B$,
given $f\in H^0(\lambda)^{U^+}$ we have
\[f(u_1cu_2)=c^{-1}f(1)\]
for all $u_1\in U^+(A)$, $c\in C(A)$, $u_2\in U(A)$, and all $A$. 
Thus $f(1)$ determines the restriction of $f$ to $U^+B$, and hence in fact determines $f$ itself as $U^+B$ is dense in $G$ by (\ref{bigcell}). 
Now $f(1)\in L_B(\lambda)$, so $\dim H^0(\lambda)^{U^+}\leq \dim L_B(\lambda)$. 
Moreover, the evaluation map $\epsilon:H^0(\lambda)\to L_B(\lambda)$ given by $f\mapsto f(1)$ is a homomorphism of $B$-modules which is injective on $H^0(\lambda)^{U^+}$. 
This implies \[H^0(\lambda)^{U^+}\subseteq H^0(\lambda)_\lambda.\]
If $\mu$ is a maximal weight of $H^0(\lambda)$ then $H^0(\lambda)_\mu\subseteq H^0(\lambda)^{U^+}\subseteq H^0(\lambda)_\lambda$, but this allows us to conclude both that $H^0(\lambda)^{U^+}=H^0(\lambda)_\lambda$ and that $\mu\leq \lambda$ for any weight $\mu$ of $H^0(\lambda)$.

Now, restricting to the Levi subgroup $M$ of $G$, we see that if $\mu$ is a weight, then so is $w_0(\mu)$ by \cite[II.1.19(1)]{Jan03}, hence $w_0(\mu)\leq \lambda$ and $w_0(\lambda)\leq \mu$.
\end{proof}

We have thus found a module $H^0(\lambda)$ of highest weight $\lambda$ whose $\lambda$-weight space is of dimension $\dim L_C(\lambda)=\dim L_B(\lambda)$ as required. In fact, one also sees:

\begin{cor}\label{simplecor}
If $H^0(\lambda)\neq 0$ then $\soc_G(H^0(\lambda))$ is simple.
\end{cor}
\begin{proof} 
If $L_1$ and $L_2$ are two simple submodules of $H^0(\lambda)$ then $L_1\oplus L_2\subseteq H^0(\lambda)$ hence $L_1^{U^+}\oplus L_2^{U^+}\subseteq H^0(\lambda)^{U^+}$ and $\dim H^0(\lambda)^{U^+}\geq 2\cdot \dim L_C(\lambda)$, contradicting Prop.~ \ref{highweightprop}(a).
\end{proof}

In order to finish, we first need to show that $\soc_G(H^0(\lambda))$ coincides with $L_G(\lambda)$ as defined in the previous section. 
Just for the following proof let us denote $\tilde L_G(\lambda):=\soc_G(H^0(\lambda))$.

\begin{lemma}\label{socsaresame} We have $\tilde L_G(\lambda)\cong L_G(\lambda)$.\end{lemma}
\begin{proof}Denote $H^0_M(\lambda):=\Ind_{B_M}^M(\lambda)$ where $B_M$ is the lower Borel subgroup of $M$. This module is the injective hull of $L_M(\lambda)$ in the category of modules with weights $\mu\leq \lambda$; see \cite[A.6]{Jan03}. Since $\tilde L_G(\lambda)$ has weights $\mu\leq\lambda$ it therefore admits a non-zero $M$-homomorphism to $H^0_M(\lambda)$. Hence, by Frobenius reciprocity we get a non-zero homomorphism from $\tilde L_G(\lambda)$ to $\Ind_M^G(H^0_M(\lambda))$. Since there is an injection $L_M(\lambda)\to H^0_M(\lambda)$, the left-exactness of the $\Ind_M^G$ functor means that there is also an injection $Q_G(\lambda)\to\Ind_M^G(H^0_M(\lambda))$. But as $L_G(\lambda)$ is the socle of $Q_G(\lambda)$ we will be done if we can show that $\Ind_M^G(H^0_M(\lambda))$ has a simple socle, for then $L_G(\lambda)$ and $\tilde L_G(\lambda)$ must both map to that simple socle. As in the proof of Theorem \ref{existanduniq}, base change everything to $k'$ and consider \[\Hom_{G_{k'}}(L_{G_{k'}}(\mu),\Ind_{M_{k'}}^{G_{k'}}(H^0_M(\lambda)_{k'}))\cong \Hom_{M_{k'}}(L_{M_{k'}}(\mu),H^0_M(\lambda)_{k'}),\] using Frobenius reciprocity. The latter has dimension $\delta_{\lambda\mu}$ over $k'$ owing to the simplicity of the socle of $H^0_M(\lambda)_{k'}\cong H^0_{M_{k'}}(\lambda)$. This proves the claim.\end{proof}

\begin{cor}The simple $G$-module $L_G(\lambda)$ is isotypic and semisimple on restriction to $M$.\end{cor}
\begin{proof}In view of Lemma \ref{socsaresame} and Prop.~\ref{highweightprop}(b) the weights $\mu$ of $L_G(\lambda)$ all satisfy $\mu\leq \lambda$. The Weyl module $V_M(\lambda)$ is the projective cover of $L_M(\lambda)$ in the category of $M$-modules with this condition on weights, by \cite[A.6]{Jan03} again. Thus it follows that any $M$-submodule of $L_G(\lambda)$ whose head is isomorphic to $L_M(\lambda)$ is the image of $V_M(\lambda)$. Suppose the image of $V_M(\lambda)$ is not precisely $L_M(\lambda)$. Then $L_G(\lambda)$ has an $M$-submodule $L_M(\mu)$ for some $\mu<\lambda$. Thus 
\[0\neq\Hom_M(L_M(\mu),L_G(\lambda))\cong\Hom_M(L_G(\lambda)^*,L_M(\mu)^*).\]
We have $L_G(\lambda)^*\cong L_G(-w_0(\lambda))$ (resp. $L_M(\mu)^*\cong L_M(-w_0(\mu))$) since the former is a simple module of weight $-w_0(\lambda)$ (resp.~$-w_0(\mu)$)---see \cite[II.2.5]{Jan03}. So using Frobenius reciprocity we get a non-trivial $G$-homomorphism from $L_G(-w_0(\lambda))\to Q_G(-w_0(\mu))$. The simplicity of the socle of $Q_G(-w_0(\mu))$ implies $\lambda=\mu$, a contradiction.

Hence, the $M$-socle of $L_G(\lambda)$ contains all $M$-composition factors of the form $L_M(\lambda)$. If there are any other composition factors, then for some $\mu<\lambda$ we have 
\[0\neq\Hom_M(L_G(\lambda),L_M(\mu))\cong\Hom_G(L_G(\lambda),Q_G(\mu));\]
again, we conclude $\lambda=\mu$, a contradiction. We have, incidentally, shown that $\Res^G_M(L_G(\lambda))$ is equal to its socle, hence is semisimple.
\end{proof} 

\begin{proof}[Proof of Theorem \ref{dimformula}] Since $\dim H^0(\lambda)_\lambda=\dim L_C(\lambda)$, $\dim L_M(\lambda)_\lambda=1$ and $L_G(\lambda)$ is $M$-isotypic with factors $L_M(\lambda)$, we are done.\end{proof}

%This gives another immediate proof of Theorem \ref{dimformula}:

%\begin{proof}[Proof of Theorem \ref{dimformula}]Let $V$ be any simple $G$-module. Then it is finite-dimensional and its restriction to $B$ admits a surjection onto a simple $B$-module $L_B(\lambda)$. But now Frobenius reciprocity together with Corollary \ref{simplecor} implies an isomorphism with $\soc(H^0(\lambda))=L_G(\lambda)$.\end{proof}

\begin{remarks}\label{rems:notsemisimple}(i). \label{abssemi}The proof of Theorem \ref{existanduniq} shows that the $G_{k'}$-module $Q_{G}(\lambda)_{k'}$ has a simple socle; hence so does its submodule $L_{G}(\lambda)_{k'}$. As $L_G(\lambda)_{k'}$ has a total of $\dim L_C(\lambda)$ $G_{k'}$-composition factors, by Theorem \ref{dimformula}, it is rarely semisimple. Hence, $L_G(\lambda)$ is usually not absolutely semisimple. This is at odds with the situation for a finite group $G$, where $k'$ is a splitting field for $G$ iff it contains all the $|G|$th roots of unity. This entails that any simple $kG$-module is split by a separable extension of $k$. Using a Galois decent argument, it follows that any semisimple $kG$-module is absolutely semisimple. See \cite[Cor.~7.11]{CR90}.

(ii). If one of the root groups of $G$ has dimension strictly bigger than one\footnote{This condition on the root groups is equivalent to the statement that $G/Z(G)\not\cong M/Z(M)$. We are thankful to Brian Conrad for confirmation of this point.} then $Q_G(\lambda)$ will not be isotypic (for $G$, or equivalently for its restriction to $M$). To see this, observe that $Q_G(\lambda)_{k'}\cong L_M(\lambda)_{k'}\otimes k'[U]$, by \cite[I.3.8(2)]{Jan03}, where $G_{k'}\cong M_{k'}\ltimes U$. Here, $U$ acts on $k'[U]$ by left translation, trivially on $L_M(\lambda)_{k'}$ and $M$ acts by conjugation on $k'[U]$. Since $U$ contains non-trivial parts of the root groups of $G$ which are not in $M$, the weights of $M_{k'}$ on $G_{k'}$-module subquotients of $U$ are not all zero. Ultimately, the same follows for $k'[U]$. Hence $k'[U]$ contains $M$-composition factors which are not trivial; but it also contains the submodule isomorphic to the trivial module generated by the constants. It follows that $L_M(\lambda)_{k'}\otimes k'[U]$ is not $M$-isotypic. It would be interesting to understand the structure of $Q_G(\lambda)$ further, even in the case that $G$ is a Weil restriction. \end{remarks}

\section{On representations of commutative pseudo-reductive groups}

Recall our assumption that $k$ is imperfect of characteristic $p$. In this section, we (amongst other things) calculate $\dim L_C(\lambda)$ whenever $C$ is the Weil restriction $\R_{k'/k}(\Gm)$ for $k'$ a non-zero finite reduced $k$-algebra whose factor fields are all purely inseparable over $k$. This assumption guarantees that $C$ is pseudo-split, so that by Theorem \ref{existanduniq}, the isomorphism classes of simple $C$-modules are in $1$--$1$ correspondence with the weights of a maximal torus of $C$.

To start with we may make some general remarks.

\subsection{Blocks of commutative pseudo-split pseudo-reductive groups}\label{blocks}
Let $C$ be a commutative pseudo-split pseudo-reductive group with maximal split torus $T$. If $V$ is any non-zero $C$-module, and $\lambda$ is any $T$-weight of $C$ on $V$ then $V_\lambda$ is a $C$-submodule, using the commutativity of $C$. By Theorem \ref{existanduniq}, $V_\lambda$ is isotypic, with composition factors all isomorphic to $L_C(\lambda)$. Indeed the projection $\mathrm{pr}_\lambda$ of $V$ to $V_\lambda$ is a $C$-module map which splits. In particular $\Ext^1_C(L_C(\lambda),L_C(\mu))\neq 0$ only if $\lambda=\mu$, so it follows that the blocks of $C$ are in bijection with $X(T)$.

\begin{remark} Any commutative (connected) reductive group $G$ is linearly reductive, hence all $\Ext_G^n(V,W)$ vanish for all $G$-modules $V$, $W$ and integers $n>0$. This in contrast to the pseudo-reductive case: for example if $G=\R_{k'/k}(\Gm)$ for $k'/k$ a purely inseparable field extension then $G$ acts on the unipotent group $G/\Gm$, hence on $k[G/\Gm]$. By (\ref{unipindecomp}), $k[G/\Gm]$ is indecomposable, and infinite dimensional. In particular $\Ext^1_G(k,k)\neq 0$.\end{remark}

\subsection{The case $\R_{k'/k}(\Gm)$ for $k'$ a field} For the time being, we let $k'$ be a purely inseparable field extension of $k$.
We recall that the exponent of such an extension is the minimal $e$ such that $k'^{p^e} \subseteq k$. Set $C = \R_{k'/k}(C')$, where $C'=\Gm$. Following (\ref{canonicalsub}) we denote by $T$ the canonical copy of $\Gm$ inside $C$. Of course, $T$ is a Levi subgroup of $C$.

\subsubsection{The standard module}
We identify $C'$ with $\GL_1$, acting faithfully on the $1$-dimensional vector group $S'\cong \Ga$. 
Applying the Weil restriction function $\R_{k'/k}$ we have that $C$ acts faithfully on the $[k':k]$-dimensional vector group $S=\R_{k'/k}(S')$. 
Let $k'$ have basis $\{1=\alpha_1,\alpha_2,\dots, \alpha_{q}\}$ as a $k$-vector space.
This choice allows us to identify $C$ with a $k$-subgroup of $\GL_{[k':k]}$, where the matrix of $g\in C(A)\cong \Gm(k'\otimes_k A)$ on $S(A)$ is calculated by acting on $S'(k'\otimes_k A)$ and taking coordinates relative to the given $k$-basis. 

Since $C'(k')=(k')^\times$ has two orbits on the vector space $S'(k')=k'$ (one trivial and one non-trivial) we see that $C(k)$ has also two orbits on $S(k)=S'(k')$. Thus we conclude that $S$ is an irreducible module for $C$ and refer to it as the \emph{standard} (or \emph{natural}) module for $C$. 
It is easy to see that the canonical subgroup $T$ of $G$ acts on $C$ as scalars; 
more precisely, if $a\in T(A)=A^\times$ for a $k$-algebra $A$, then $a\cdot s=as$ for all $s\in S(A)$. 
In other words, $\Res^C_T S=(k_1)^{\oplus[k':k]}$, the direct sum of $[k':k]$ copies of the weight space $k_1$ with weight $1$. Of course, in light of Theorem \ref{existanduniq} we must have

\begin{equation}\label{sisl1}S\cong L_C(1),\end{equation}

and so $\dim L_C(1)=\dim S=[k':k]$.

Following \cite[I.2.15]{Jan03} we may twist any representation of any algebraic group $G$ by precomposing with an endomorphism $\sigma$ of $G$. If $V$ is a $G$-module then we denote the resulting representation by ${}^\sigma V$. This gives rise to a functor $G$-$\Mod\to G$-$\Mod$ by $V\mapsto\sigma^*(V)={}^\sigma V$. For $\lambda \in \Z$, by precomposing the representation of $C'$ on $S'$ by the function $\sigma_\lambda':C'\to C'$ via $c\mapsto c^\lambda$, we get a representation $(\sigma_\lambda')^*(S')$ on which $C'$ acts with weight $\lambda$, indeed $(\sigma_\lambda')^*(S')\cong k_\lambda$. Taking the Weil restriction of this representation then gives the representation $\sigma_\lambda^*(S)$ where $\sigma_\lambda:C\to C$ via $c\mapsto c^\lambda$ also.

If $p\nmid\lambda$ it is quite easy to see that $(\sigma_\lambda)^*(S)$ is irreducible, giving us the dimension in this case, but we will show something stronger, namely that $\sigma_\lambda^*$ is an equivalence of categories, in fact, that $\sigma_\lambda^*$ has an inverse $\tau$. In order to do this, we will want to understand the coordinate algebra $k[C]$ a little better. We prove more than we need and give a complete description.

\subsubsection{Description of $k[\R_{k'/k}(\Gm)]$}\label{subsec:k[C]} 
Recall our choice of $k$-basis $\{1=\alpha_1,\alpha_2,\dots, \alpha_{q}\}$ of $k'$. 
The coordinate algebra of $C'$ is $k'[C'] = k'[t,t^{-1}]$.
To find the coordinate algebra of $C = \R_{k'/k}(C')$, we should rewrite the 
generators $t$ and $t^{-1}$ in terms of our chosen $k$-basis of $k'$.
So we introduce new functions $\hat\alpha_i$ on $C$ and write $t =  \sum_{i=1}^q\alpha_i \hat\alpha_i$.
Equivalently, when we identify $C$ with a subgroup of $\GL_{[k':k]}$ via its (left) action on $S$,
the $\hat\alpha_i$ can be identified with the matrix coordinate functions from the first column,
i.e., for $g\in C(A)$, $\hat\alpha_i(g)$ gives the coefficient of $\alpha_i$ in $g\cdot 1$. 
If $e$ is the exponent of the extension $k'/k$ then the $p^e$-power map takes $(k')^\times$ into $k^\times$. 
Hence it takes $C$ into its canonical copy of $\Gm$, which we have denoted $T$. 
We thus get a $1$-dimensional representation $\hat d$ of $C$, i.e. an element of the character group $X(C)$. 
The function $\hat d$ is a polynomial of degree $p^e$ in the functions $\hat\alpha_i$. 
Given any $k$-algebra $A$ and any $g\in C(A)$, $g^{p^e}$ is represented by a scalar matrix with diagonal entries all equal to $\hat d(g) \in A^\times$.
Let $\det$ denote the element of $k[C]$ given by taking determinants of matrices---then we see that
for any $g \in C(A)$ we have 
$$
\det(g^{p^e}) = (\hat d(g))^q,
$$ 
so $(\det)^{p^e} = \hat d^q$.
We conclude that the function $\det$ is a power of the function $\hat d$ in $k[C]$.
Now we can write $t^{-1}=t^{p^e-1}/t^{p^e}$ in terms of the $\hat\alpha_i$ and the function $\hat d^{-1}$. Evidently $\hat\alpha_i$ and $\hat d^{-1}$ are elements of $k[C]$.

\begin{prop}\label{gensofkg}
The natural map
\[F:k[\hat\alpha_i,\hat d^{-1}]_{1\leq i\leq q}\to k[C],\] is an isomorphism.

As a $C$-module, the action of $T$ induces a grading on the generators so that $\hat\alpha_i$ is in degree $1$ and $\hat d$ is in degree $p^e$.
\end{prop}

To see that $F$ is an isomorphism it suffices to see that it is an isomorphism after extension to $\bar k$. We have $C_{\bar k}\cong \Gm\times U$, where $U$ is the unipotent radical of $C_{\bar k}$ and the quotient $\Gm$ of $C_{\bar k}$ corresponds to the subalgebra $\bar k[\hat\alpha_1,\hat\alpha_1^{-1}]$ of $\bar k[C]$. Since $U$ is a connected unipotent algebraic group we have $\bar k[U]\cong \bar k[\beta_2,\dots,\beta_q]$, for some indeterminates $\beta_i$ in the image of the comorphism $\bar k[C_{\bar k}]\to\bar k[U]$. %Since $U$ is also a quotient of $C_{\bar k}$ we can identify $\bar k[U]$ as a subalgebra of $\bar k[C]$. 
We can find appropriate choices for the $\beta_i$ from the following, which uses the natural embedding of $C$ in $\A_q$:

\begin{lemma}
Let $A$ be a $\bar k$-algebra and $a=(a_1,\dots,a_q)\in C_{\bar k}(A)\subseteq \A_q(A)$. Then $a\in U(A)$ if and only if $a_1=1-\sum_{2\leq i\leq q}a_i\alpha_i$.\end{lemma}
\begin{proof} It is an elementary calculation using the multiplication inherited from $(k')^\times$ as a $k$-group to see that the given condition on $a$ defines a subgroup all of whose elements are unipotent. Since it is evidently $(q-1)$-dimensional, the claim follows.\end{proof}

\begin{proof}[Proof of Proposition \ref{gensofkg}]
%Consider the short exact sequence $1\to\Gm \to C_{\bar k}\to U\to 1$.
%The lemma shows that if $I_U$ denotes the augmentation ideal in ${\bar k}[U]$,
%then the image of $I_U$ in ${\bar k}[C_{\bar k}]$ generates the ideal ${\bar k}[C_{\bar k}](\hat\alpha_2,\dots,\hat\alpha_q)$ in ${\bar k}[C_{\bar k}]$,
%and hence we have an exact sequence
%\[
%0 
%\to 
%{\bar k}[C_{\bar k}](\hat\alpha_2,\dots,\hat\alpha_q)
%\to 
%{\bar k}[C_{\bar k}]
%\to 
%R:={\bar k}[\hat\alpha_1,\hat\alpha_1^{-1}]\to 0.\] 
%But the images of $\hat\alpha_1$ and $\hat d^{-1}$ in $R$ are $\hat\alpha_1$ and $\hat\alpha_1^{-p^e}$ respectively, so we can conclude that $F$ is an isomorphism as claimed.
Consider the short exact sequence $1\to U \stackrel{i}\to C_{\bar k}\stackrel{\mathsf q}\to \Gm\to 1$, where $\mathsf q$ is the canonical map described in \S\ref{sec:Weil}. At the level of $k$-algebras, this corresponds to
\[0\to I_{\Gm}\cdot\bar k[C_{\bar k}]\stackrel{q^*}\to \bar k[C_{\bar k}]\stackrel{i^*}\to \bar k[U]\to 0,\]
where $I_{\Gm}$ is the image in ${\bar k}[C_{\bar k}]$ of the augmentation ideal in $\bar k[\Gm]$. The lemma implies that $\bar k[U]$ is the quotient of $\bar k[C_{\bar k}]$ by the ideal generated by $\hat\alpha_1-1+\sum_{2\leq i\leq q}\alpha_i\hat \alpha_i=\hat d^{p^{-e}}-1$; thus $I_{\Gm}\cdot{\bar k}[C_{\bar k}]=(\hat d^{p^{-e}}-1){\bar k}[C_{\bar k}]$, and $\bar k[\Gm]=\bar k[\hat d^{p^{-e}},\hat d^{-p^{-e}}]$. Since the values of the $\hat\alpha_i$ on a point of $C_{\bar k}$ determine it completely, and $U$ is isomorphic to affine $(q-1)$-space, we may choose $\beta_i=i^*(\hat\alpha_i)$ for $2\leq i\leq q$. Now, $R:=\bar k[\hat\alpha_1,\dots,\hat\alpha_q,\hat d^{-1}]=\bar k[\hat\alpha_1,\dots,\hat\alpha_q,\hat d^{-p^{-e}}]$:  the inclusion $\subseteq$ is obvious, and for the other direction, note $\hat d^{p^{-e}}\in R$. But as $C_{\bar k}$ is a semidirect product, we have \begin{align*}\bar k[C_{\bar k}]&\cong\bar k[\Gm]\otimes \bar k[\A_{q-1}]\\
&=\bar k[\hat d^{p^{-e}},\hat d^{-p^{-e}}]\otimes \bar k[\hat \alpha_2,\dots,\hat \alpha_q]\\
&=\bar k[\hat \alpha_1,\dots,\hat\alpha_q,\hat d^{-p^{-e}}]\\
&=R.\end{align*} 

%, so the map $\bar F$ gives rise to the left and right isomorphisms
%\[\begin{CD}
%0@>>> (\hat d^{1/q}-1)\cdot R@>>> R @>>> \bar k[\hat\alpha_2,\dots,\hat\alpha_q]@>>> 0\\
%@. @V\cong VV @VVV @V\cong VV @.\\
%0@>>> I_{\Gm}\cdot\bar k[C_{\bar k}]@>>> \bar k[C_{\bar k}]@>>> \bar k[U]@>>> 0
%\end{CD}\]

It remains to prove the second part. We define a grading by setting $\deg(f)=j$ if $z\cdot f=z^{-j}f$ for any $z\in T(k)=k^\times$. By definition $\hat\alpha_i(g)$ is the coefficient of $\alpha_i$ in $g\cdot 1$, so that $z\cdot \hat\alpha_i(g)=\hat\alpha_i(z^{-1}g)$ gives the coefficient of $\alpha_i$ in $z^{-1}g\cdot 1$. But this is clearly just $z^{-1}\hat\alpha_i(g)$, and so $\hat\alpha_i$ is in degree $1$. Now if $f=gh$ is a product of functions in $k[C]$, then $f(z^{-1}g)=g(z^{-1}g)h(z^{-1}g)$ so that a monomial in $j$ of the $\hat\alpha_i$'s is in degree $j$. Lastly, $\hat d$ is by definition a function for which $z\cdot\hat d=z^{-p^e}\hat d$ so that $\hat d$ is in degree $p^e$.\end{proof}

\begin{remark} In \cite[\S1]{Fri10} a general method for computing the coordinate rings of Weil restrictions is given. Applied to non-trivial separable field extensions this generally gives a much more complicated presentation than the one we have shown above. For example, if one takes $C'=\Gm$, $k=\mathbb{R}$ and $k'=\mathbb{C}$, then $k'[C']=k'[x,y]/(xy-1)$ and $k[\R_{k'/k}(C')]=k[x_1,x_2,y_1,y_2]/(x_1y_1-x_2y_2-1,x_2y_1+y_2x_1)$.\end{remark}

For the benefit of the reader, we give some concrete examples to illustrate the constructions above.

\begin{example}
(i). Suppose $p=3$ and $k'/k$ is a purely inseparable extension of fields of degree $3$.
Write $k' = k(s)$ so that $k'$ has $k$-basis $1,s,s^2$.
Then the matrix representation of $G$ gives matrices of the form:
$$
\left(\begin{array}{ccc}
a & cs^3 & bs^3\\
b & a & cs^3\\
c & b & a
\end{array}\right).
$$
One can check that the cubing map sends such a matrix to a diagonal matrix with diagonal entries $a^3+b^3s^3+c^3s^6$.
Thus, if the functions $\hat\alpha_1, \hat\alpha_2, \hat\alpha_3$ are the coordinate functions for the first column, 
the function $\hat d = \hat\alpha_1^3 + s^3\hat\alpha_2^3 + s^6\hat\alpha_3^3$.
One can also check that in this case we get the same function by taking determinants of these matrices.
Now $k'[G'] = k'[t,t^{-1}]$ and we can write $t = \hat\alpha_1 +s\hat\alpha_2 +s^2\hat\alpha_3$.
Then we calculate:
$$
t^{-1} = \frac{1}{\hat\alpha_1 +s\hat\alpha_2 +s^2\hat\alpha_3} = \frac{(\hat\alpha_1 +s\hat\alpha_2 +s^2\hat\alpha_3)^2}{\hat d},
$$
which is a polynomial function of the $\hat\alpha_i$ and $\hat d^{-1}$.

(ii). Suppose $p=2$ and $k'/k$ is a purely inseparable extension of degree $4$ such that 
$k' = k[s,u]$ with $s^2,u^2\in k\setminus k^2$.
Then $k'$ has $k$-basis $1,s,u,su$ and the extension has exponent $e=1$.
This time our matrices have the form:
$$
\left(\begin{array}{cccc}
a & bs^2 & cu^2 & ds^2u^2\\
b & a & du^2 & cu^2\\
c & ds^2 & a & bs^2\\
d & c & b & a
\end{array}\right).
$$
One checks that the function $\hat d=\hat\alpha_1^2 + s^2\hat\alpha_2^2 + u^2\hat\alpha_3^2 + s^2u^2\hat\alpha_4^2$. The function coming from the determinant is $\hat d^2$.
\end{example}

\subsubsection{Irreducible modules $L(\lambda)$ with $p\nmid\lambda$}
We can now define an inverse to $\sigma_\lambda^*$ when $p\nmid \lambda$. 
Let $\hat d$ be the $1$-dimensional representation from Section \ref{subsec:k[C]}. 
Recall that $\hat d$ arises from the $p^e$-power map for $e$ the exponent of $k'/k$, 
and we have $(p^e,\lambda)=1$. 
This means that we may take $\mu\in \Z$ to be an inverse to $\lambda$ in their projections to $\Z/(p^e)$,
i.e., we can choose $\mu$ so that $\lambda\mu=1+rp^e$ for some $r$.
Then set $\tau$ to be the composition of $\sigma^*_{\mu}$ with the tensor product functor $?\otimes \hat d^{-r}$. 
Thus $\tau(V)=\sigma^*_\mu(V)\otimes \hat d^{-r}$. 

\begin{prop}
The functors $\tau$ and $\sigma:=\sigma_\lambda^*$ are mutually inverse.
\end{prop}

\begin{proof} 
%We have already shown that $\sigma$ and $\sigma^*_{\mu}$ take irreducible modules of dimension $q$ to irreducible modules of dimension $q$. 
%Tensoring by the module $d^{-r}$ is obviously exact, and so $\tau$ also takes irreducible modules of dimension $q$ to irreducible modules of dimension $q$. 
%
To see this, simply note that the map $C\to C; x\mapsto x^{\lambda\mu}x^{-rp^e}$ is the identity on $C$. Since $C$ acts on $(\tau\circ\sigma)(V)$ as $x^{\lambda\mu}x^{-rp^e}$ acts on $V$ (considering $x^{-rp^e}$ as an element of $k$), it follows that $\tau\circ\sigma$ and $\sigma\circ\tau$ are both the identity functor on $C$-$\Mod$ as required.
\end{proof}

We conclude from above that $\tau$ and $\sigma$ are equivalences of categories, i.e. Morita equivalences. 
Since the latter sends representations of weight $1$ to representations of weight $\lambda$, of the same dimension, we deduce that for $p\nmid\lambda$,

\begin{equation}\sigma^*_\lambda(L_C(1))\cong L_C(\lambda)\end{equation}
\begin{equation}\label{pnmiddim}\dim L_C(\lambda)=[k':k]\end{equation}

\subsubsection{Irreducible modules $L(\lambda)$ with $p\mid \lambda$}\label{subs:irredpdiv}  
Let $\tilde k$ be the field generated by $k$ and the $p^{\rm th}$ powers in $k'$ (this is the field $k'(p)$ in the notation introduced before the statement of Theorem \ref{gmfield}). 
Then $\tilde k$ is in fact the $k$-span of the $p^{\rm th}$ powers in $k'$. 
We see $\tilde C:=\R_{\tilde k/k}(C')$ naturally as a subgroup of $C$ because $\tilde k\subseteq k'$. 
We also have a quotient  $C_1$ of $C$ which we can realise as a subgroup of $\tilde C$. Let $F:\Gm \to \Gm$ denote the geometric Frobenius map (with comorphism $t\mapsto t^p$). 
Then since $k'^p \subseteq \tilde{k}$ we can see that 
$\R_{k'/k}(F):C\to \tilde C$. Let $C_1$ be the image of this map;  
in general $C_1$ is a $k$-subgroup of $\tilde C$, but may not be all of $\tilde C$.

We have a $k$-basis $\alpha_1,\dots,\alpha_q$ of $k'$, so $\alpha_1^p,\dots,\alpha_q^p$ span $\tilde k$ over $k$. It follows that the group algebra $kC_1(k)=k\tilde C(k)$. Since also the $k$-points of each of $C_1$ and $C$ are dense in those groups, we deduce \begin{equation}\text{ a }\tilde C\text{-module }V\text{ is irreducible if and only if it is irreducible on restriction to }C_1\label{irredonrest}.\end{equation} (Both groups of course are defined over $k$, hence $V$ is in particular a $k$-module by definition.)

We now prove the counterparts to the last section.

\begin{prop}\label{pdiv} Let $\lambda\in X(T)$.
We have $L_C(p\lambda)\cong (\R_{k'/k}(F))^*(L_{\tilde C}(\lambda))$. Furthermore, 
\[\dim L_C(\lambda)=[k'(\lambda):k].\]
\end{prop}
\begin{proof}Since $L_{\tilde C}(\lambda)$ is irreducible, (\ref{irredonrest}) implies that $\R_{k'/k}(F)^*(L_{\tilde C}(\lambda))$ is irreducible for $C$. Since $F^*(k_\lambda)\cong k_{p\lambda}$ as $T$-modules, the first statement follows.

For the second, we proceed by induction on $[k':k]$. If $p\nmid\lambda$ then we are done by (\ref{pnmiddim}), otherwise, let $\lambda=p\lambda'$. Then by induction, $\dim L_{\tilde C}(\lambda')=[\tilde k(\lambda'):k]=[k'(p)(\lambda'):k]=[k'(\lambda):k]$, and we are done by the first part.
\end{proof}

\subsection{The case $\R_{k'/k}(\Gm)$ for $k'$ a non-zero finite reduced $k$-algebra}

With the results of the previous section in hand, we can now give a dimension formula in the case where $C=\R_{k'/k}(\Gm)$ for $k'$ a non-zero finite reduced $k$-algebra.
This is precisely the case where $C=\prod_{1\leq i\leq n}\R_{k_i/k}(\Gm)$, and the $k_i$ are the factor fields of $k'$. Since we insist $C$ is pseudo-split, we must have that the $k_i$ are all purely inseparable extensions of $k$ and so $k'$ is a purely inseparable $k$-algebra. By \cite[A.7.8]{CGP15} the minimal field of definition of the unipotent radical of each factor $\R_{k_i/k}(\Gm)$ is $k_i$ itself. Since each $k_i$ is a finite field extension it embeds into the algebraic closure $\bar k$ and the unique minimal field of definition $K'$ for $\RR_u(G_{\bar k})$ is the subfield of $\bar k$ generated by the $k_i$.

Let $T$ be the canonical Levi subgroup $\prod_{1\leq i\leq n}\Gm$ inside $C$. By Theorem \ref{existanduniq} there is up to isomorphism a unique simple $C$-module $L_C(\lambda)$ of weight $\lambda=(\lambda_1,\dots,\lambda_n)\in X(T)\cong \Z^n$. The following theorem calculates its dimension:

\begin{theorem}We have $\dim L_C(\lambda)=[K:k]$, where $K$ is the subfield of $K'$ generated by $k$ together with $k_i(\lambda_i)$ for each $1\leq i\leq n$.\label{thm:anykalgebra}\end{theorem}
\begin{proof}We must exhibit an irreducible module of weight $\lambda$ and the correct dimension. For each $i$ write $\lambda_i=p^{e_i}\mu_i$ where $p\nmid \mu_i$ and put $\tilde k_i=k_i(p^{e_i})$. Since $k_i/k$ is purely inseparable we must have $\tilde k_i=k_i(\lambda_i)$ and hence $K$ is generated by the $\tilde k_i$. From \S\ref{subs:irredpdiv} recall there is a map $\R_{k_i/k}(F^{e_i}):\R_{k_i/k}(\Gm)\to \R_{\tilde k_i/k}(\Gm)$. Precomposing a module $M$ of weight $\nu$ for $\R_{\tilde k_i/k}(\Gm)$ with $\R_{k_i/k}(F^{e_i})$ gives a module $\R_{k_i/k}(F^{e_i})^*(M)$ of weight $p^e\nu$ for $\R_{k_i/k}(\Gm)$.

Now by definition of $K$, we have each $\tilde k_i$ a subfield of $K$ and hence we get an embedding $\iota_i:\R_{\tilde k_i/k}(\Gm)\hookrightarrow \R_{K/k}(\Gm)$ for each $i$. We also get a map $\R_{K/k}(\mu):\prod_{1\leq i\leq n}\R_{K/k}(\Gm)\to \R_{K/k}(\Gm)$ corresponding to the weight $\mu=(\mu_1,\dots,\mu_n)$.

Now form the composite map

\begin{equation}\label{eq:bigX}
\mathcal{X}:G\cong\prod_i \R_{k_i/k}(\Gm)\xrightarrow{\prod_i\R_{k_i/k}(F^{e_i})}\prod_i \R_{\tilde k_i/k}(\Gm)\xrightarrow{\prod_i\iota_i}\prod_i \R_{K/k}(\Gm)\xrightarrow{\R_{K/k}(\mu)} \R_{K/k}(\Gm).
\end{equation}

Then if $S$ is the natural module for $\R_{K/k}(\Gm)$ one sees easily that $\mathcal{X}^*(S)$ is a module for $C$ of weight $\lambda$. We must see that $\mathcal{X}^*(S)$ is irreducible; in other words, $\mathcal{X}^*(S)\cong L_C(\lambda)$.

But as $S$ is irreducible for $\R_{K/k}(\Gm)(k)$ we need only see that $\mathcal{X}$ induces a surjection of group algebras $k(C(k))\to k(\R_{K/k}(\Gm))(k)=K$. This follows essentially from the definition of $K$: we have $k\mathcal{X}(C)(k)$ contains $\tilde k_i=k\R_{\tilde k_i/k}(\Gm)(k)=k\tilde k_i^\times$ for each $i$ and so $k\mathcal{X}(C)(k)=K$ as required.
\end{proof} 

\begin{remark} We are indebted to the referee for the following observation. Starting from a pseudo-split $k$-group $C$ of the form $\R_{k'/k}(\Gm)$ for $k'$ a finite non-zero reduced $k$-algebra, one may recover the $k$-algebra $k'$ up to (usually non-unique) isomorphism. For, let the factor fields of $k'$ be $k_1,\dots,k_r$. Then if $K$ is a field which is minimal subject to the requirement that $C_K$ contains a direct $\Gm$-factor, then without loss of generality, $K$ contains $k_1$ say, and as $C_{k_1}$ has a direct $\Gm$-factor, we have $K=k_1$ by minimality. Thus $C$ has a direct factor $\R_{k_1/k}(\Gm)$. Taking the quotient of $C$ by this factor yields another $k$-group of the same form and we are done by induction.\end{remark}

\subsection{Rank 1 pseudo-split pseudo-reductive commutative groups}
Let $C$ be a pseudo-split pseudo-reductive commutative group with a maximal split torus $T$.
If $k'$ is the minimal field of definition of the unipotent radical of $C$ then we may form $\mathscr{C}:=\R_{k'/k}(C_{k'})\cong \R_{k'/k}(T_{k'}\ltimes \RR_{u,k'}(C_{k'}))=\R_{k'/k}(T_{k'})\ltimes U$, where $U=\R_{k'/k}(\RR_{u,k'}(C_{k'}))$ is a unipotent normal subgroup of $\mathscr{C}$. By the usual properties of Weil restriction, $C$ embeds as a canonical $k$-subgroup of $\mathscr{C}$. Since $U$ acts trivially on any simple $\mathscr{C}$-module, and since $T$ can be naturally identified as a maximal torus in $\mathscr{C}$, in view of Theorem \ref{existanduniq} we have natural correspondences
\[\text{simple }C\text{-modules}\longleftrightarrow X(T)\longleftrightarrow\text{simple }\R_{k'/k}(T_{k'})\text{-modules} \longleftrightarrow\text{simple }\mathscr{C}\text{-modules},\]
and thus we may proceed by assuming that $C$ is a subgroup of $D:=\R_{k'/k}(T_{k'})$. Now, we know the dimensions of the simple $D$-modules, by the work of the previous sections, so it is natural to ask how these modules restrict to $C$.
In this section we give an answer when the dimension of $T$ is $1$.
Let us start with a lemma:

\begin{lemma}\label{lem:stabsubfield}
Suppose $k'/k$ is a finite extension and let $G = \R_{k'/k}(\Gm)$.
Let $W$ be a proper non-trivial $k$-subspace of $k'$, and let $S = \Stab_G(W)$.
Then there exists an intermediate field $k\subseteq E \subsetneq k'$ such that $S\subseteq \R_{E/k}(\Gm)$. 
\end{lemma}

\begin{proof}
First recall that for the subspace $W$ we have the corresponding (additive) group functor
$W_a$ given by $W_a(A) := W\otimes_k A$ (see Section \ref{sec:Gmods}). 
Then $S$ is an algebraic subgroup scheme of $G$ and for every $k$-algebra $A$ we have, by \cite[I.2.12]{Jan03}
$$S(A) = \{g\in G(A)\mid gW_a(A)\subseteq W_a(A)\}.$$
Now, since all multiplication here is commutative, given any nonzero $w\in W = W_a(k)$ we can see that 
$S$ also stabilizes the $k$-subspace $w^{-1}W$ of $k'$.
Hence, we may assume that $1\in W$.

Now let $E$ denote the intersection of all $k$-subspaces of $k'$ containing $1$ and stabilized by $S$
(i.e., the intersection of all $k$-subspaces $X$ such that $1\in X$ and $S(A)(X\otimes_k A) \subseteq X\otimes_k A$ for all $k$-algebras $A$).
We show that $E$ is an intermediate field and that $S\subseteq E_a$.
Given this, $S$ is then contained in the corresponding multiplicative unit group, which is precisely $\R_{E/k}(\Gm)$.

To see that $E$ is a field we just need to show $E\setminus\{0\}$ is a group under multiplication.
So let $x\in E\setminus\{0\}$.
Appealing to the commutativity of multiplication again, we see that the $k$-subspace $x^{-1}E$ is $S$-stable and contains $x^{-1}x = 1$, so we must have $x^{-1}E = E$, as required.
Since $W$ is proper, $E$ is properly contained in $k'$.
The final step is to note that since $1\in E$ and $E$ is $S$-stable, 
we have $S(A)\subseteq E_a(A)$ for all $k$-algebras $A$. Hence $S\subseteq E_a$.
\end{proof}

\begin{prop} Let $C$ be a pseudo-split commutative pseudo-reductive group of rank $1$ and let $k'$ be the minimal field of definition of its unipotent radical. Set $D=\R_{k'/k}(\Gm)$. Then the restriction of $L_D(\lambda)$ to $C$ is irreducible.\end{prop} 

\begin{proof}
Pick a weight $\lambda = p^e\mu$. Then the irreducible $L_{D}(\lambda)$ identifies with the field $k'(\lambda) = k(p^e)$ inside $k'$, and this contains a copy of the irreducible $L_C(\lambda)$.
Now, by the argument in the proof of Lemma \ref{lem:stabsubfield}, we may assume that $L_C(\lambda)$ contains some $C$-stable subfield $E$; but $L_C(\lambda)$ is irreducible, so $E=L_C(\lambda)$. 
Therefore, the image of $C$ under the representation is contained in $\R_{E/k}(\Gm)$. 
Let $E' \subseteq k'$ be the set of $x$ such that $x^{p^e} \in E$ (the preimage of $E$ under Frobenius).

Then $E'$ is a subfield of $k'$ containing $k$ because $E$ is a subfield of $k'$ and $k \subseteq E'$ obviously. Also, $C$ is contained in $\R_{E'/k}(\Gm)$, so actually $E' = k'$ by minimality of $k'$. But then $E$ must be $k'(\lambda)$ by definition of $k'(\lambda)$.
\end{proof}

\begin{remark}
As a closing remark, we note that similar arguments using the observations at the start of this section and Lemma \ref{lem:stabsubfield} show that for any commutative pseudo-split pseudo-reductive
group $C$, the irreducible module $L_C(\lambda)$ is a subfield of the minimal field of definition of the unipotent radical of $C$.
Similar results have been proved by Brion \cite[Prop.~3.1]{Bri18} in a slightly different context, and using a different approach.
\end{remark}

\bigskip
{\bf Acknowledgements}:
We would like to acknowledge very helpful discussions with Ben Martin, particularly around the material in Section \ref{sec:exist}.
We also thank Brian Conrad and Gopal Prasad for comments and discussion. 
Lastly, we thank the anonymous referees for their careful reading of earlier versions and their suggestions for improvements.

{\footnotesize
\bibliographystyle{amsalpha}
\bibliography{bib}}

\providecommand{\bysame}{\leavevmode\hbox to3em{\hrulefill}\thinspace}
\providecommand{\MR}{\relax\ifhmode\unskip\space\fi MR }
% \MRhref is called by the amsart/book/proc definition of \MR.
\providecommand{\MRhref}[2]{%
  \href{http://www.ams.org/mathscinet-getitem?mr=#1}{#2}
}
\providecommand{\href}[2]{#2}
\begin{thebibliography}{CGP15}

\bibitem[Bri18]{Bri18}
Michel Brion, \emph{Homogeneous vector bundles over abelian varieties via
  representation theory}, arxiv:1805.09531 (2018), preprint.

\bibitem[CGP15]{CGP15}
Brian Conrad, Ofer Gabber, and Gopal Prasad, \emph{Pseudo-reductive groups},
  second ed., New Mathematical Monographs, vol.~26, Cambridge University Press,
  Cambridge, 2015. \MR{3362817}

\bibitem[CP17]{CPSurv}
Brian Conrad and Gopal Prasad, \emph{Structure and classification of
  pseudo-reductive groups}, Algebraic groups: structure and actions, Proc.
  Sympos. Pure Math., vol.~94, Amer. Math. Soc., Providence, RI, 2017,
  pp.~127--276. \MR{3645069}

\bibitem[CR90]{CR90}
Charles~W. Curtis and Irving Reiner, \emph{Methods of representation theory.
  {V}ol. {I}}, Wiley Classics Library, John Wiley \& Sons, Inc., New York,
  1990, With applications to finite groups and orders, Reprint of the 1981
  original, A Wiley-Interscience Publication. \MR{1038525}

\bibitem[Fri10]{Fri10}
Eric~M. Friedlander, \emph{Weil restriction and support varieties}, J. Reine
  Angew. Math. \textbf{648} (2010), 183--200. \MR{2774309}


\bibitem[Jan03]{Jan03}
J.~C. Jantzen, \emph{Representations of algebraic groups}, second ed.,
  Mathematical Surveys and Monographs, vol. 107, American Mathematical Society,
  Providence, RI, 2003. \MR{MR2015057 (2004h:20061)}

\bibitem[L{\"u}b01]{Lub01}
Frank L{\"u}beck, \emph{Small degree representations of finite {C}hevalley
  groups in defining characteristic}, LMS J. Comput. Math. \textbf{4} (2001),
  135--169 (electronic). \MR{1901354 (2003e:20013)}

\bibitem[RW18]{RW18}
Simon Riche and Geordie Williamson, \emph{Tilting modules and the
  {$p$}-canonical basis}, Ast\'{e}risque (2018), no.~397, ix+184. \MR{3805034}

\bibitem[SGA3]{SGA3}
Philippe Gille and Patrick Polo (eds.), \emph{Sch\'{e}mas en groupes ({SGA} 3).
  {T}ome {III}. {S}tructure des sch\'{e}mas en groupes r\'{e}ductifs},
  Documents Math\'{e}matiques (Paris) [Mathematical Documents (Paris)], vol.~8,
  Soci\'{e}t\'{e} Math\'{e}matique de France, Paris, 2011, S\'{e}minaire de
  G\'{e}om\'{e}trie Alg\'{e}brique du Bois Marie 1962--64. [Algebraic Geometry
  Seminar of Bois Marie 1962--64], A seminar directed by M. Demazure and A.
  Grothendieck with the collaboration of M. Artin, J.-E. Bertin, P. Gabriel, M.
  Raynaud and J-P. Serre, Revised and annotated edition of the 1970 French
  original. \MR{2867622}

\bibitem[Tit71]{Tit71}
J.~Tits, \emph{Repr{\'e}sentations lin{\'e}aires irr{\'e}ductibles d'un groupe
  r{\'e}ductif sur un corps quelconque}, J. Reine Angew. Math. \textbf{247}
  (1971), 196--220. \MR{0277536}

\bibitem[Wil17]{Wil17}
Geordie Williamson, \emph{Schubert calculus and torsion explosion}, J. Amer.
  Math. Soc. \textbf{30} (2017), no.~4, 1023--1046, With a joint appendix with
  Alex Kontorovich and Peter J. McNamara. \MR{3671935}

\end{thebibliography}
\end{document}